\documentclass[11pt,a4paper]{amsart}
\usepackage{graphicx,multirow,array,amsmath,amssymb}

\newtheorem{theorem}{Theorem}

\begin{document}

\title[Bisected vertex leveling of plane graphs]
{Bisected vertex leveling of plane graphs: braid index, arc index and delta diagrams}

\author[S. No]{Sungjong No}
\address{Institute of Mathematical Sciences, Ewha Womans University, Seoul 03760, Korea}
\email{sungjongno84@gmail.com}

\author[S. Oh]{Seungsang Oh}
\address{Department of Mathematics, Korea University, Seoul 02841, Korea}
\email{seungsang@korea.ac.kr}

\author[H. Yoo]{Hyungkee Yoo}
\address{Department of Mathematics, Korea University, Seoul 02841, Korea}
\email{lpyhk727@korea.ac.kr}

\thanks{The corresponding author(Seungsang Oh) was supported by the National Research Foundation of Korea(NRF) grant funded by the Korea government(MSIP) (No. NRF-2017R1A2B2007216).}
\thanks{The first author was supported by Basic Science Research Program through the National Research Foundation of Korea (NRF) funded by the Ministry of Education(2009-0093827).}

\maketitle

\begin{abstract}
In this paper, we introduce a bisected vertex leveling of a plane graph.
Using this planar embedding, we present elementary proofs of the well-known
upper bounds in terms of the  minimal crossing number on braid index $b(L)$ and arc index $\alpha(L)$ 
for any knot or non-split link $L$, which are
$b(L) \leq \frac{1}{2} c(L) + 1$ and $\alpha(L) \leq c(L) + 2$.
We also find a quadratic upper bound of the minimal crossing number of delta diagrams of $L$.
\end{abstract}

\section{Introduction} \label{sec:introduction}

In knot theory people have developed a variety of ways 
to represent knots and links in specific conformations. 
And a presentation of a knot or link accompanies a minimal quantity 
which is necessary for representing a knot or link type in a particular conformation. 
For example, the braid index and arc index are the minimal number of 
strings and arcs to make braid and arc presentations, respectively.

There are many references on braid index and arc index \cite{BP, Cr, CN, LNO, Ohy, R, Ya}, 
nevertheless it is not easy to determine the exact value in general.
Instead people are interested in their upper and lower bounds. 
Since knots and links are tabulated according to the minimal crossing number, 
usually the bounds are written in terms of the minimal crossing number.

Let $L$ be any knot or non-split link and $c(L)$ denote the minimal crossing number of $L$. 
A known bound for the braid index $b(L)$ of $L$ is
$$ b(L) \leq \frac{1}{2} c(L) + 1, $$
which was proved by Ohyama~\cite{Ohy}. 
And Bae and Park~\cite{BP} gave an upper bound for the arc index $\alpha(L)$,
$$ \alpha(L) \leq c(L) + 2. $$

In this paper we propose a new presentation of a plane graph 
which will be called a {\em bisected vertex leveling\/}, 
and prove that every connected plane graph without any loops or cut-vertices has a bisected vertex leveling. 
Note that a diagram of a knot or non-split link can be regarded as a 4-valent connected plane graph
by ignoring the under/over information of crossings. 
If the diagram is of minimal crossing number, 
then the corresponding graph has no loops or cut-vertices. 
So we can use the proposed presentation of graphs when dealing with 
the combinatorics of link diagrams with minimal crossings.

Using the bisected vertex leveling,
we give simple and elementary proofs of Ohyama's upper bound in Section~\ref{sec:braid}
and Bae-Park's upper bound in Section~\ref{sec:arc}.

In Section~\ref{sec:delta}, we apply our argument to delta diagrams. 
A diagram of a knot or link is called a {\em delta diagram\/} if it consists of only 3, 4 and 5-sided regions. 
Recently Jablan, Kauffman and Lopes~\cite{JKL} found 
a quartic upper bound for the minimal crossing number on delta diagrams of 
a nontrivial knot or non-split link by using the braid presentation. 
Applying the bisected vertex leveling we find a new quadratic upper bound.

\section{Bisected vertex leveling} \label{sec:leveling}

In this section, what is meant by a plane graph is a geometric realization of a planar graph
as a finite 1-dimensional CW-complex in $\mathbb{R}^2$.
Let $G$ be a connected plane graph with $n$ vertices.
We assume that $G$ does not have loops (edges with both endpoints at a single vertex).

A {\em vertex leveling\/} of $G$ is
its ambient isotopy lying between the two horizontal lines $y=0$ and $y=n$ in $\mathbb{R}^2$
such that the new graph $G$ satisfies the following two conditions (see Figure~\ref{fig:leveling});
\begin{itemize}
\item[(1)] All vertices of $G$ lie one by one on the lines $y=k-\frac{1}{2}$, $k=1,2, \dots, n$,
and we label them by $v_k$'s.
\item[(2)] Each edge of $G$ has no maxima and minima as critical points
of the height function given by the $y$-coordinate, except its endpoints (vertices).
\end{itemize}

\begin{figure}[h]
\includegraphics{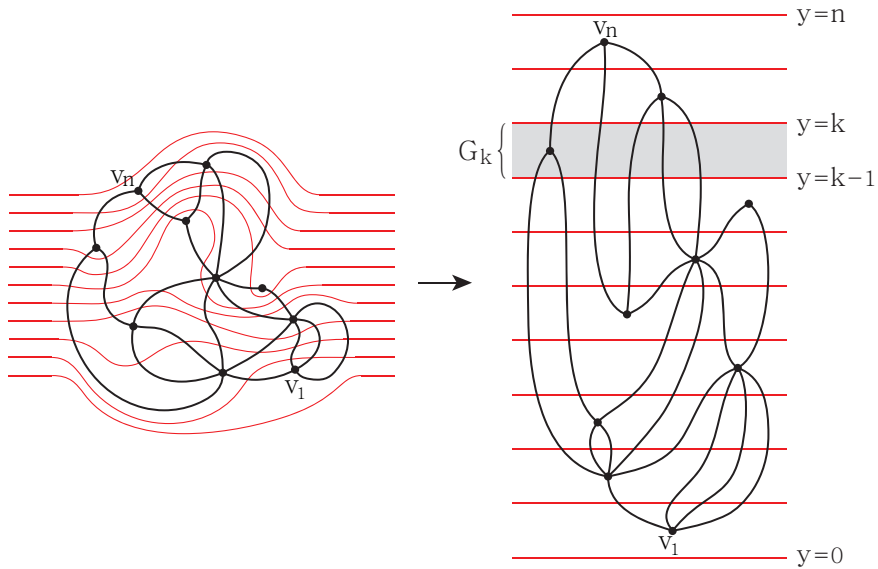}
\caption{Vertex leveling of $G$}
\label{fig:leveling}
\end{figure}

Let $G_k$ be a portion of $G$ lying between the two lines $y=k-1$ and $y=k$.
From condition (2), $G_k$ consists of three kinds of arcs;
$p$ arcs going down from $v_k$ to the line $y=k-1$,
$q$ arcs going up from $v_k$ to the line $y=k$,
and the remaining arcs going up from the line $y=k-1$ to $y=k$ as in Figure~\ref{fig:type}.
We say this has $T^q_p$ {\em type\/}.
Note that there is no arc whose two ends reach the same line.

\begin{figure}[h]
\includegraphics{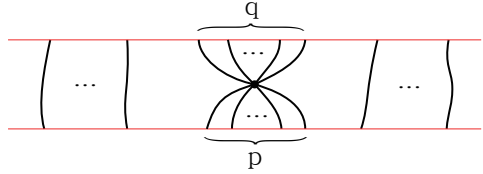}
\caption{$T^q_p$ type on each level}
\label{fig:type}
\end{figure}

It is noteworthy that the conditions (1) and (2) naturally imply the following conditions;
\begin{itemize}
\item[($3^\ast$)] Each portion $G_k$ of $G$ lying between $y=k-1$ and $y=k$
has the shape of $T^q_p$ type for some nonnegative integers $p$ and $q$.
\item[($4^\ast$)] $G_1$ has $T^q_0$ type with only $q$ arcs adjacent to $v_1$,
and similarly $G_n$ has $T^0_p$ type with only $p$ arcs adjacent to $v_n$.
\end{itemize}

A vertex leveling of $G$ is said to be {\em bisected\/} if it satisfies an additional condition as follows;
\begin{itemize}
\item[(5)] Each line $y=k$, $k=1,2, \dots, n-1$, cuts $G$ into two pieces, each of which is connected.
\end{itemize}
Then, it naturally satisfies the following condition (see Figure~\ref{fig:bvl});
\begin{itemize}
\item[($6^\ast$)] Each portion $G_k$, $k=2,3, \dots, n-1$ has neither $T^0_p$ type nor $T^q_0$ type.
\end{itemize}

The three additional conditions ($3^\ast$), ($4^\ast$) and ($6^\ast$) will be used in
Sections~\ref{sec:braid}, \ref{sec:arc} and \ref{sec:delta}.

\begin{figure}[h]
\includegraphics{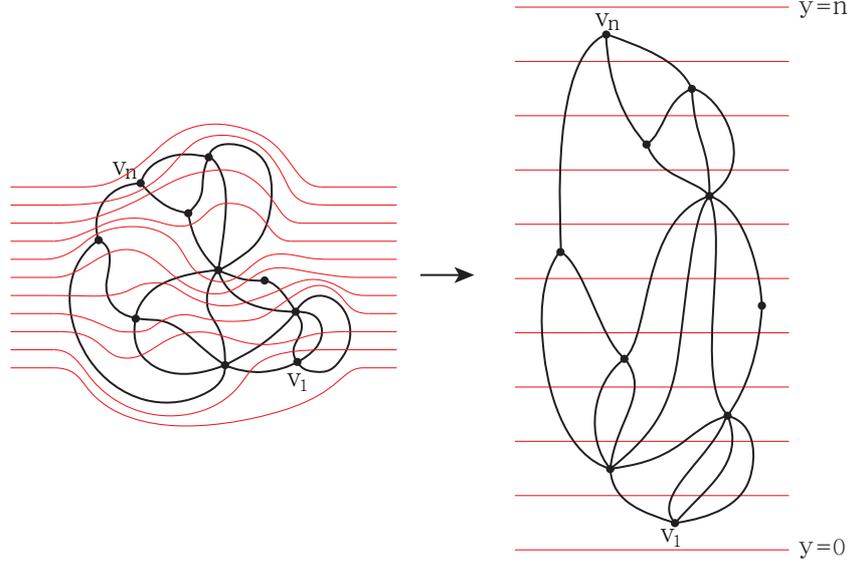}
\caption{Bisected vertex leveling of $G$}
\label{fig:bvl}
\end{figure}

We introduce the main result of this section.
A vertex $v$ is called a cut vertex of $G$
if $G \setminus \{v\}$ (deleting $v$ and its adjacent edges) has more connected components than $G$.

\begin{theorem}\label{thm:bvl}
Let $G$ be a connected plane graph without loops.
If it has no cut vertex, then $G$ has a bisected vertex leveling.
\end{theorem}

\begin{proof}
Let $G$ be any connected plane graph without loops which has $n$ vertices.
Assume that $G$ has no cut vertices.
There are at least two vertices meeting the unbounded region.
Choose two such vertices and name them $v_1$ and $v_n$.

We construct a bisected vertex leveling of $G$ by using induction on $k$.
To begin, we put the vertex $v_1$ on the line $y= \frac{1}{2}$ and move up the rest of the graph
by an ambient isotopy of $G$ so that
the portion $G_1$ lying under the line $y=1$ has $T^q_0$ type with only $q$ arcs adjacent to $v_1$.
This is possible since $v_1$ meets the unbounded region.
Note that $G \setminus \{v_1\}$ is still connected and lies above $y=1$.
Thus the line $y=1$ cuts $G$ into two connected pieces.

Suppose that we take an ambient isotopy of $G$
so that the new graph $G$ satisfies the three conditions (1), (2) and (5) of being
a bisected vertex leveling under the line $y=k$ (for $k \leq n-2$).
More precisely,
(1) some vertices $v_1, \dots, v_k$ of $G$ lie one by one on the lines $y=i-\frac{1}{2}$, $i=1,2, \dots, k$;
(2) each edge of $G$ has no maxima and minima under the line $y=k$, except its endpoints;
(5) each line $y=i$, $i=1,2, \dots, k$, cuts $G$ into exactly two connected pieces.

We will choose a vertex $v_{k+1}$ satisfying these three conditions.
Our candidates for $v_{k+1}$ are the vertices which are adjacent to some of $v_1, \dots, v_k$,
i.e., at least one of the adjacent edges of each candidate vertex must go below the line $y=k$.
Obviously, we exclude $v_n$ for candidates.
It is guaranteed that the set of candidates is nonempty;
for otherwise, either $G$ is not connected or $v_n$ must be a cut vertex of $G$.

Let $w_1$ be a vertex among the candidates.
Suppose that the part of $G \setminus \{w_1\}$ lying above $y=k$ is not connected
as in Figure~\ref{fig:choosing}.
Let $D_1$ be one component that does not contain $v_n$ and $D'_1$ the union of the other components.
Notice that both $D_1$ and $D'_1$ must contain some vertices adjacent to $w_1$
because of the condition (5) for the line $y=k$.
The component $D_1$ contains another candidate vertex $w_2$;
for otherwise, $w_1$ is a cut vertex of $G$.
Suppose further that the part of $G \setminus \{w_2\}$ lying above $y=k$ is not connected.
Then we repeat the same argument above.
Let $D_2$ be one component that does not contain $w_1$ and $D'_2$ the union of the other components.
Obviously, $D_2$ is wholly contained in $D_1$ and so $D'_2$ contains $w_1$ and $D'_1$ (and also $v_n$).

\begin{figure}[h]
\includegraphics{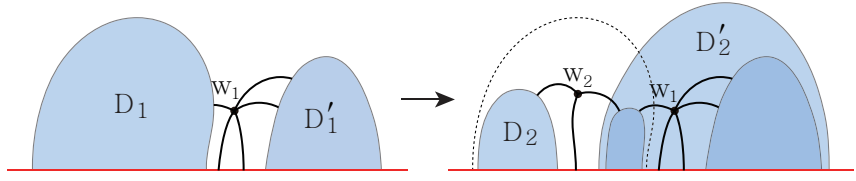}
\caption{Choosing $v_{k+1}$}
\label{fig:choosing}
\end{figure}

Choose another candidate vertex $w_3$ in $D_2$, and repeat the same argument
until we find a candidate vertex $w$ such that the part of $G \setminus \{w\}$ lying above $y=k$ is connected.
Since $G$ has finitely many vertices, we can always find one after a finite number of steps.
Denote this vertex $v_{k+1}$.

Now, we ambient isotope the part of $G$ lying above the line $y=k$ 
while holding it fixed below the line as in Figure~\ref{fig:making}.
After this isotopy, $v_{k+1}$ lies on the line $y=k+ \frac{1}{2}$ and
the portion of $G$ lying between the two lines $y=k$ and $y=k+1$, say $G_{k+1}$,
has $T^q_p$ type for some positive integers $p$ and $q$.
It is crucial that two integers are positive since
at least one of the edges adjacent to $v_{k+1}$ goes below $y=k$
and the part of $G$ lying above $y=k$ is connected.

Now the new graph $G$ satisfies the three conditions (1), (2) and (5) of being
a bisected vertex leveling under the line $y=k+1$.

At the last step of the induction, the vertex $v_{n-1}$ will be suitably chosen.
Note that the vertex $v_n$ defined at the beginning can be easily moved
so that the part of $G$ lying above the line $y=n-1$ has $T^0_p$ type for some positive integer $p$.
\end{proof}

\begin{figure}[h]
\includegraphics{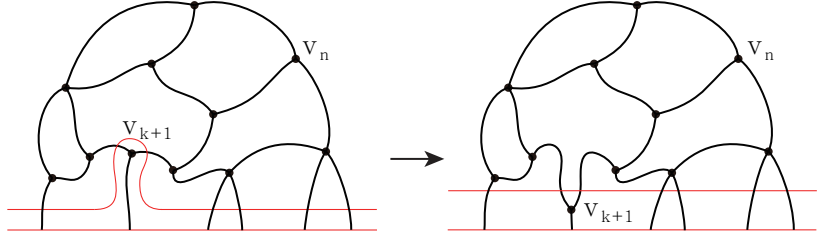}
\caption{Making $G_{k+1}$}
\label{fig:making}
\end{figure}

\section{Braid index} \label{sec:braid}

Braids were first considered as a tool of knot theory by Alexander~\cite{Al},
who proved that every knot or link $L$ can be represented as a closed braid with finitely many strings.
The {\em braid index\/} of $L$, denoted $b(L)$, is defined as the minimal number of strings
needed for $L$ to be represented as a closed braid.
Yamada~\cite{Ya} found an upper bound for the braid index by showing that $b(L) \leq S(D)$,
where $S(D)$ is the number of Seifert circles of a diagram $D$ of $L$.
Improving Yamada's result, Murasugi and Przytycki~\cite{MP} showed that $b(L) \leq S(D) - \text{ind}(D)$
by introducing a new concept $\text{ind}(D)$, called an index of $D$.

In~\cite{Ohy}, Ohyama found the relation between the crossing number $c(D)$ of 
a reduced diagram $D$ and $S(D)$
which is $c(D) \geq 2 \{S(D) - \text{ind}(D) - 1\}$.
Applying the result of Murasugi and Przytycki, Ohyama proved the following theorem.

\begin{theorem}[Ohyama] \label{thm:braid}
Let $L$ be any knot or non-split link and $c(L)$ the minimal crossing number of $L$.
Then we have an upper bound of the braid index of $L$
$$ b(L) \leq \frac{1}{2} c(L) + 1. $$
\end{theorem}

Cheng and Jin~\cite{CJ} found an infinite family of links, each of which satisfies the equality of the theorem.

In this section, a simple proof of Theorem~\ref{thm:braid} is given by means of
a bisected vertex leveling of a 4-valent plane graph which is a diagram of $L$ with minimal crossings.

\begin{proof}
Let $L$ be any knot or non-split link.
Given a diagram $D$ of $L$ with minimal $c(L)$ crossings,
one can ignore which strand is the overstrand at each crossing and
think of it as a connected 4-valent plane graph with $n=c(L)$ vertices, denoted $G$.
We assume that $G$ does not have loops and cut vertices from the fact that
the corresponding diagram has minimal crossings and so has no nugatory crossing.

By Theorem~\ref{thm:bvl}, there exists a bisected vertex leveling of $G$
so that each $G_k$, $k=2,3, \dots, n-1$ has a type among $T^3_1$, $T^2_2$ and $T^1_3$,
while $G_1$ has $T^4_0$ type and $G_n$ has $T^0_4$ type.
See the Figure~\ref{fig:braidbvl} for an example.

\begin{figure}[h]
\includegraphics{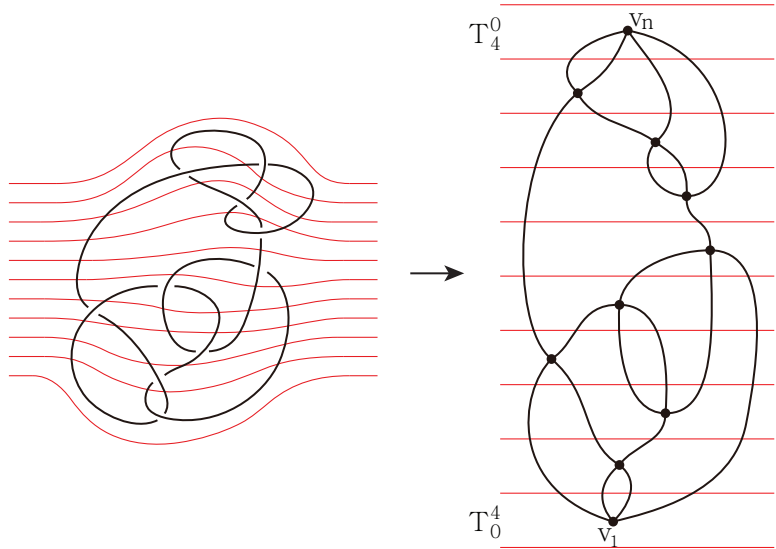}
\caption{Bisected vertex leveling of a knot}
\label{fig:braidbvl}
\end{figure}

First, replace each portion $G_k$ with a combination of horizontal and vertical line segments
as shown in Figure~\ref{fig:braidtrans}.
Then, restore the diagram $D$ from $G$ by giving the under/over information on each vertex of $G$
according to the original crossing information of $D$.
See the left picture in Figure~\ref{fig:braidrestor}.

\begin{figure}[h]
\includegraphics{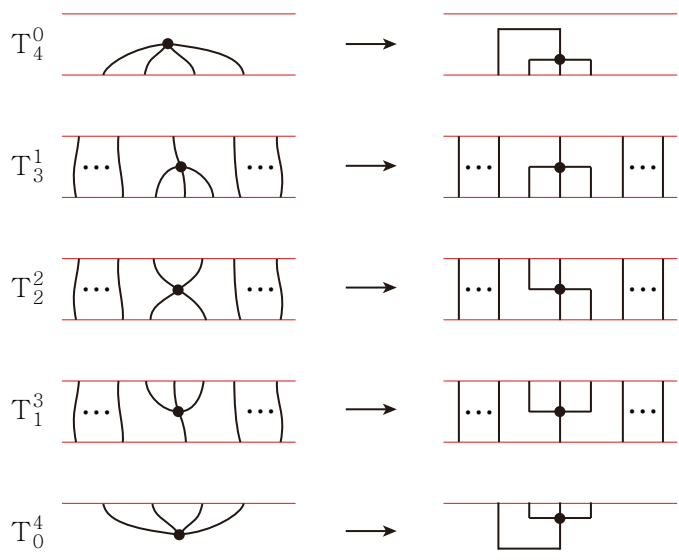}
\caption{Transforming $G_k$}
\label{fig:braidtrans}
\end{figure}

\begin{figure}[h]
\includegraphics{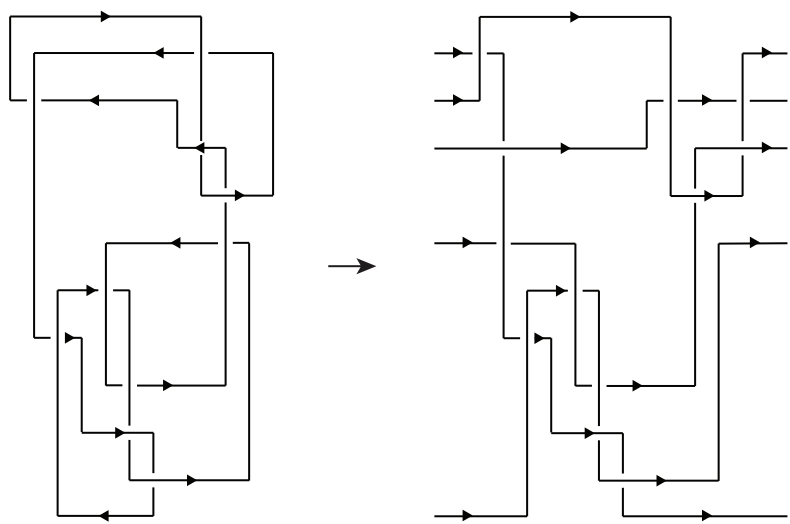}
\caption{Restoring $D$ from $G$ and making a braid}
\label{fig:braidrestor}
\end{figure}

It is noteworthy that there are exactly $n+2$ horizontal line segments and
each of them, except the top and the bottom horizontal line segments,
crosses exactly one vertical line segment.

Choose one direction on $D$ to make a directed diagram.
If $L$ is a link with more than one component, then choose a direction for each component.
Each horizontal line segment has two possibilities, say left directed or right directed.
Without loss of generality,
we assume that the number of left directed horizontal line segments are
less than or equal to the number of right directed horizontal line segments.

Replace each left directed horizontal line segment
with two external horizontal rays starting at its two endpoints as follows;
if it runs through under (resp., over) a vertical line segment,
then the two external rays must run through under (resp., over) other vertical line segments
as the right picture in Figure~\ref{fig:braidrestor}.
This is always possible since each horizontal line segment crosses at most one vertical line segment.
This picture can be viewed as a braid which represents $L$ as in Figure~\ref{fig:braidform}.

\begin{figure}[h]
\includegraphics{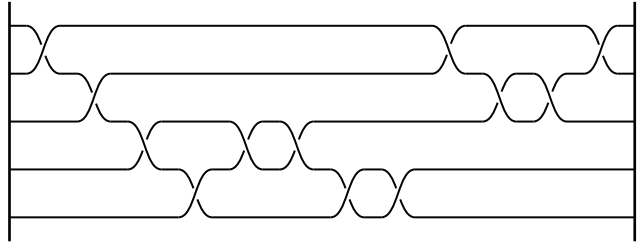}
\caption{Braid form of $L$}
\label{fig:braidform}
\end{figure}

Since there are at most $\frac{n+2}{2}$ left directed horizontal line segments,
the braid index of $L$ is at most $\frac{n+2}{2}$.
\end{proof}

\section{Arc index} \label{sec:arc}

The three-dimensional space $\mathbb{R}^3$ has an open-book decomposition 
which has open half-planes as pages and the standard $z$-axis as the binding axis.
In an {\em arc presentation\/} of a knot or link $L$, 
it is embedded in an open-book with finitely many pages 
so that it meets each page in exactly one simple arc with two different end-points on the binding axis.
Here the points of $L$ on the binding axis are labeled  by $1, 2, \dots$ in order,
which are called the {\em binding indices\/}.

Cromwell~\cite{Cr} introduced the term {\em arc index\/} $\alpha(L)$ to denote
the minimal number of arcs to make an arc presentation of $L$.
Later, Cromwell and Nutt~\cite{CN} conjectured the following upper bound on the arc index
in terms of the minimal crossing number which was proved by Bae and Park~\cite{BP}.

\begin{theorem}[Bae-Park] \label{thm:arc}
Let $L$ be any knot or non-split link and $c(L)$ the minimal crossing number of $L$.
Then we have an upper bound of the arc index of $L$
$$ \alpha(L) \leq c(L) + 2. $$
\end{theorem}

Using the bisected vertex leveling argument, we present another elementary proof of this theorem.

\begin{proof}
Let $L$ be any knot or non-split link.
By following the argument of the first four paragraphs in the proof of Theorem~\ref{thm:braid},
we use the diagram $D$ depicted as the left picture in Figure~\ref{fig:braidrestor}.
Note that $D$ consists of $c(L) + 2$ horizontal line segments and $c(L) + 2$ vertical line segments,
and each horizontal line segment crosses at most one vertical line segment.
In this case, we do not need to give a direction to the diagram.

We assign $1, 2, \dots, c(L)+2$ from left to right for the vertical line segments.
We also denote the horizontal line segments by $b_k$, $k=1,2, \dots, c(L)+2$, from bottom to top.
The diagram $D$ can always be converted to an arc presentation of $L$ as follows 
(see Figure~\ref{fig:arcpresent}).
First, shrink the vertical line segment labeled by each number $k$ to a point
which is indeed the $k$-th binding index.

Next, transform each horizontal line segment $b_k$ to an arc connecting the two binding indices
which correspond to the two vertical line segments adjacent to $b_k$.
We arrange these $c(L)+2$ arcs according to the following two rules;
\begin{itemize}
\item[(a)] The arcs related to the horizontal line segments running over some vertical line segment
lie above the binding axis, and the other arcs lie below the binding axis.
\item[(b)] If $i < j$, then the arc $b_j$ lies behind the arc $b_i$ when they cross.
\end{itemize}

\begin{figure}[h]
\includegraphics{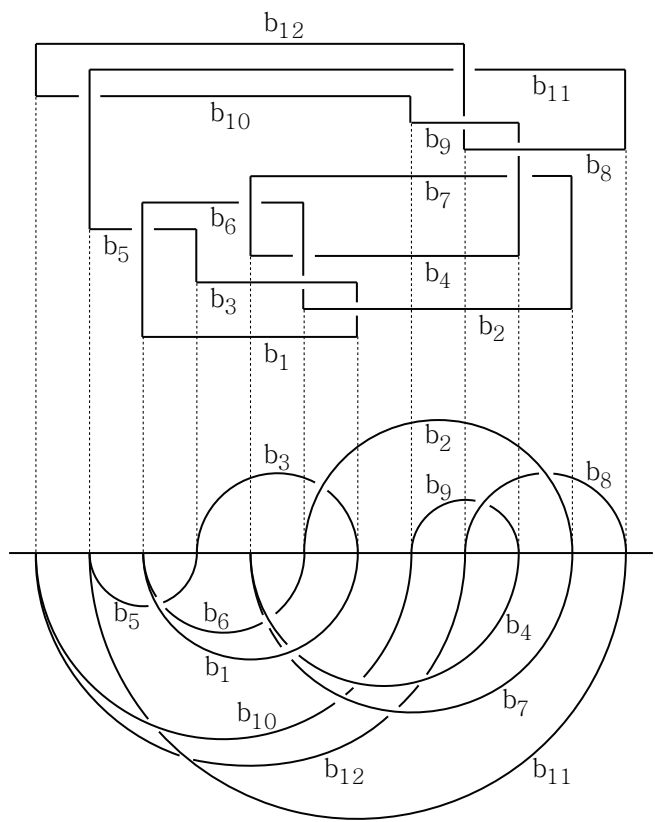}
\caption{An arc presentation of $L$}
\label{fig:arcpresent}
\end{figure}

Indeed, if we pull $b_k$'s which run over some vertical line segment forward and push the rest $b_k$'s backwards, 
then we get the bottom picture in Figure~\ref{fig:arcpresent} when we look at it from the bottom.

Now rotate each arc lying above the binding axis $180^{\circ}$ along the binding axis 
so that it does not pass through other arcs.
Finally we get an arc presentation of $L$ with $c(L)+2$ arcs.
\end{proof}

\section{Delta diagrams} \label{sec:delta}

In this section we investigate the graphical structure of diagrams of links
which have a restriction to the number of edges of each region.
It is easy to show that every link has a diagram
which does not possess regions with 1 edge by simply untwisting the nugatory crossings.

There are several important results published recently.
Eliahou, Harary and Kauffman~\cite{EHK} proved that
every link has a diagram which does not possess regions with 2 sides (lunes).
This diagram is called a {\em lune-free diagram\/}.
For example, see the left picture in Figure~\ref{fig:lune-delta}.
Adams, Shinjo and Tanaka~\cite{AST} introduced an increasing sequence of integers
which is said to be {\em universal\/} if every link has a diagram such that
the number of sides of each region (including the unbounded one) comes from the given sequence.
They provided several universal infinite sequences such as
$(3, 5, 7, \dots ), (2, n, n+1, n+2, \dots )$ for each $n \geq 3$ and $(3, n, n+1, n+2, \dots )$ for each $n \geq 4$,
and universal finite sequences such as $(2, 4, 5)$ and $(3, 4, n)$ for each $n \geq 5$.

\begin{figure}[h]
\includegraphics{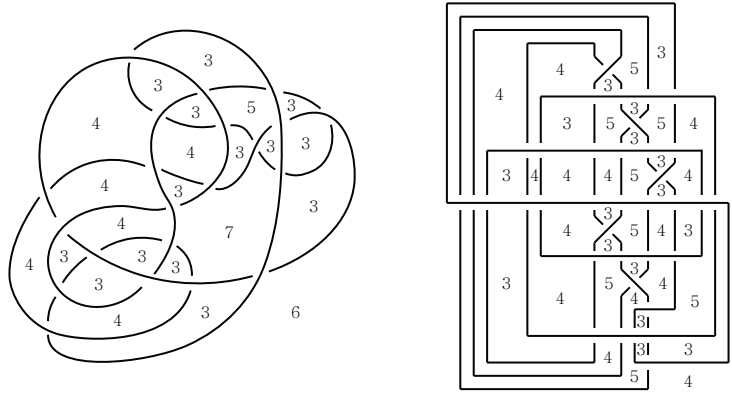}
\caption{A lune-free diagram and a delta diagram}
\label{fig:lune-delta}
\end{figure}

Recently Jablan, Kauffman and Lopes~\cite{JKL} independently showed that
any link can be represented by a diagram
whose regions possess 3, 4 or 5 sides, called a {\em delta diagram\/}.
In that paper, they start with a braid closure representation of the link
and deform it in order to obtain a desired delta diagram as the right picture in Figure~\ref{fig:lune-delta}.
They also proved that every link has a delta diagram with at most
$c(L)^4-4c(L)^3+12c(L)^2-16c(L)+15$ crossings in~\cite[Theorem 3.2]{JKL}.

In this section we present a quadratic upper bound on the number of crossings
produced by the transformation into a delta diagram.

\begin{theorem} \label{thm:delta}
Let $L$ be a nontrivial knot or non-split link (except the Hopf link)
and $c(L)$ the minimal crossing number of $L$.
Then $L$ has a delta diagram with at most $\frac{1}{2}c(L)^2 + 4c(L) - 5$ crossings.
\end{theorem}

\begin{proof}
Let $L$ be a nontrivial knot or non-split link, which is not  the Hopf link.
As mention in the first two paragraphs of the proof of Theorem~\ref{thm:braid},
$G$ is a connected 4-valent plane graph with $n=c(L)$ vertices obtained from
a minimal crossing diagram of $L$,
and there exists a bisected vertex leveling of $G$
so that each $G_k$, $k=2,3, \dots, n-1$ has a type among $T^3_1$, $T^2_2$ and $T^1_3$,
while $G_1$ has $T^4_0$ type and $G_n$ has $T^0_4$ type.
We may assume that $n$ is at least 3.

From this bisected vertex leveling of $G$,
we will obtain a delta diagram of $L$ by applying the following procedure.
See Figure~\ref{fig:bvl-delta} for the resulting diagram
obtained from the bisected vertex leveling in Figure~\ref{fig:braidbvl}.
Start by restoring the diagram $D$ of $L$ by giving the under/over information on each vertex of $G$.
Take $n-1$ parallel line segments
which are parts of the lines $y=i$, $i=1,2, \dots, n-1$, and connect them in a spiral form, say $S$.

\begin{figure}[h]
\includegraphics{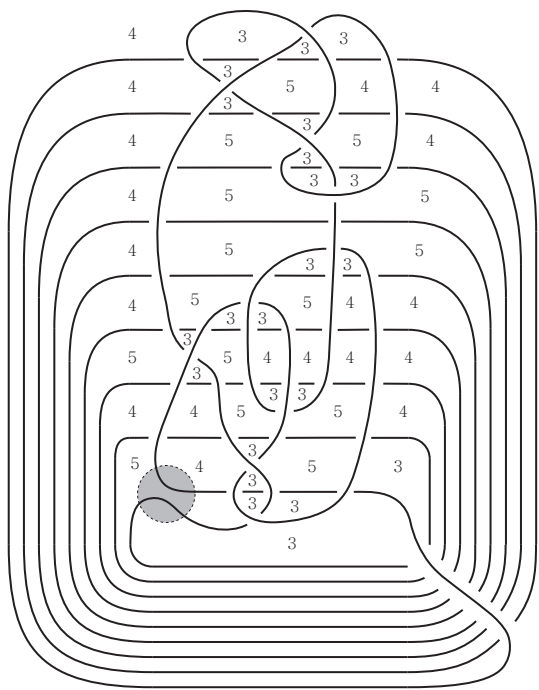}
\caption{Delta diagram obtained from a bisected vertex leveling}
\label{fig:bvl-delta}
\end{figure}

Place $S$ under $D$.
Obviously, the resulting diagram $S \cup D$ has regions with only 3, 4 or 5 sides.
In the case the vertex $v_2$ of $G_2$ is located at the leftmost side,
then we begin after rotating the diagram $D$ $180^{\circ}$ along the $y$-axis so that
$v_2$ lies at the rightmost side in $G_2$.
Now we join them near the leftmost intersection on the line $y=1$ as in the picture 
(see the shaded circle).
The result is a diagram of $L$, and it has regions with only 3, 4 or 5 sides as in Figure~\ref{fig:join}.
In the picture, we illustrate the change of the number of sides of the regions near the intersection
in two cases: $n=3$ and $n \geq 4$.

\begin{figure}[h]
\includegraphics{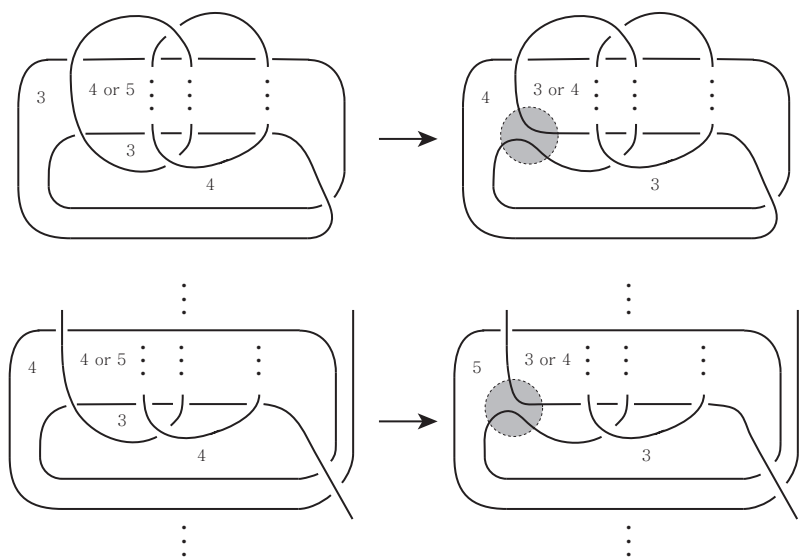}
\caption{Joining $S$ and $D$}
\label{fig:join}
\end{figure}

Finally, we calculate the number of crossings which come from the following three cases;
\begin{itemize}
\item[(1)] The diagram $D$ has $n$ self-crossings,
\item[(2)] The spiral $S$ has $n-2$ self-crossings,
\item[(3)] $D$ and $S$ meet at less than or equal to $\frac{1}{2} n^2 + 2n - 3$ crossings.
\end{itemize}
The countings of crossings in the cases (1) and (2) are obvious from the picture.
In the case (3), $D$ and $S$ meet at the most points when the lower half of $G_k$'s have $T^3_1$ type
and the upper half of $G_k$'s have $T^1_3$ type.
More precisely, if for odd $n$, $G_2, \dots, G_{\frac{n-1}{2}}$ have $T^3_1$ type,
$G_{\frac{n+1}{2}}$ has $T^2_2$ type,
and $G_{\frac{n+3}{2}}, \dots, G_{n-1}$ have $T^1_3$ type,
then they meet at $\big( 2 \sum^{\frac{n+1}{2}}_{k=2} 2k \big) -1$ points;
if for even $n$, $G_2, \dots, G_{\frac{n}{2}}$ have $T^3_1$ type
and $G_{\frac{n+2}{2}}, \dots, G_{n-1}$ have $T^1_3$ type,
then they meet at $\big( 2 \sum^{\frac{n}{2}}_{k=2} 2k \big) + (n+2) - 1$ points.
Notice that the subtraction by $1$ comes from the deletion of the point joining $D$ and $S$.
This concludes the proof.
\end{proof}

\end{document}